\documentclass[12pt]{amsart}
\usepackage{amssymb}
\usepackage[all]{xy}
\usepackage{mathrsfs}
\newtheorem{theorem}{Theorem}[]
\newtheorem{lemma}[theorem]{Lemma}
\newtheorem{proposition}[theorem]{Proposition}
\newtheorem{definition}[theorem]{Definition}
\newtheorem{question}[theorem]{Question}

\newtheorem{corollary}[theorem]{Corollary}
\newtheorem{problem}[theorem]{Problem}
\def\to{\rightarrow}
\def\f{\mathfrak}
\def\c{\mathcal}

\def\r{\mathrm}
\def\bb{\mathbb}

\def\sl{\langle}
\def\sr{\rangle}

\begin{document}
%\baselineskip15pt
%%%%%%%%%%%%%%%%%%%%%%%%%%%%%%%%%%%%%%%%%%%%%%%%%%%%%%%%%%%%%%%%%%%%%%%%%%%%%%%%%%%%%%%%%%%%%%%%%%%%%%%%%%%%%%%%%%%%%%%%%%
\title[Hyperspaces associated to a noncommutative space]{Vietoris topology on hyperspaces associated to a noncommutative compact space}
%%%%%%%%%%%%%%%%%%%%%%%%%%%%%%%%%%%%%%%%%%%%%%%%%%%%%%%%%%%%%%%%%%%%%%%%%%%%%%%%%%%%%%%%%%%%%%%%%%%%%%%%%%%%%%%%%%%%%%%%%%
\author[M. M. Sadr]{Maysam Maysami Sadr}
\address{Department of Mathematics\\
Institute for Advanced Studies in Basic Sciences\\
P.O. Box 45195-1159, Zanjan 45137-66731, Iran}
\email{sadr@iasbs.ac.ir}
%%%%%%%%%%%%%%%%%%%%%%%%%%%%%%%%%%%%%%%%%%%%%%%%%%%%%%%%%%%%%%%%%%%%%%%%%%%%%%%%%%%%%%%%%%%%%%%%%%%%%%%%%%%%%%%%%%%%%%%%%%
\subjclass[2010]{46L05; 46L85; 54B20.}
\keywords{C*-algebra, state space, closed projection, hyperspace, Vietoris topology, Hausdorff distance, infimum distance.}
%%%%%%%%%%%%%%%%%%%%%%%%%%%%%%%%%%%%%%%%%%%%%%%%%%%%%%%%%%%%%%%%%%%%%%%%%%%%%%%%%%%%%%%%%%%%%%%%%%%%%%%%%%%%%%%%%%%%%%%%%%
\begin{abstract}
We study some topological spaces that can be considered as hyperspaces associated to
noncommutative spaces. More precisely, for a NC compact space associated
to a unital C*-algebra, we consider the set of closed projections of the second dual of the C*-algebra
as the hyperspace of closed subsets of the NC space. We endow this hyperspace with an analog of Vietoris topology.
In the case that the NC space has a quantum metric space structure
in the sense of Rieffel we study the analogs of Hausdorff and infimum distances on the hyperspace.
We also formulate some interesting problems about distances between sub-circles of a quantum torus.
\end{abstract}
%%%%%%%%%%%%%%%%%%%%%%%%%%%%%%%%%%%%%%%%%%%%%%%%%%%%%%%%%%%%%%%%%%%%%%%%%%%%%%%%%%%%%%%%%%%%%%%%%%%%%%%%%%%%%%%%%%%%%%%%%%
\maketitle
%%%%%%%%%%%%%%%%%%%%%%%%%%%%%%%%%%%%%%%%%%%%%%%%%%%%%%%%%%%%%%%%%%%%%%%%%%%%%%%%%%%%%%%%%%%%%%%%%%%%%%%%%%%%%%%%%%%%%%%%%%
%%%%%%%%%%%%%%%%%%%%%%%%%%%%%%%%%%%%%%%%%%%%%%%%%%%%%%%%%%%%%%%%%%%%%%%%%%%%%%%%%%%%%%%%%%%%%%%%%%%%%%%%%%%%%%%%%%%%%%%%%%
%%%%%%%%%%%%%%%%%%%%%%%%%%%%%%%%%%%%%%%%%%%%%%%%%%%%%%%%%%%%%%%%%%%%%%%%%%%%%%%%%%%%%%%%%%%%%%%%%%%%%%%%%%%%%%%%%%%%%%%%%%
%%%%%%%%%%%%%%%%%%%%%%%%%%%%%%%%%%%%%%%%%%%%%%%%%%%%%%%%%%%%%%%%%%%%%%%%%%%%%%%%%%%%%%%%%%%%%%%%%%%%%%%%%%%%%%%%%%%%%%%%%%
\section{Introduction}\label{s1}
This note is a contribution to Noncommutative Topology. We introduce and study some topological spaces that can be considered as
the hyperspaces associated to noncommutative spaces. More precisely, let $\f{q}A$ denote the (imaginary) NC compact Hausdorff space
associated to a unital C*-algebra $A$. We consider the set of nonzero closed projections of the second dual
$A^{**}$ as the hyperspace $\f{S}_\r{cl}\f{q}A$ of nonempty closed subsets of $\f{q}A$. In the case that $A$ is commutative these closed
projections are canonically identified with closed subsets of the Gelfand space of $A$.
(To the best of our knowledge the study of closed projections as closed subsets of NC spaces goes back to Akemann \cite{Akemann1,Akemann2,Akemann3}.
Closed projections have been considered also in some recent papers, see \cite{BlecherNeal2,Comman2} and references therein.)
There is a canonical bijection between closed projections in $A^{**}$ and weak*-closed faces of
the state space $\c{S}A$ of $A$ (see Section \ref{s4}). Thus we can identify the hyperspace $\f{S}_\r{clc}\c{S}A$ of such subsets of $\c{S}A$
with $\f{S}_\r{cl}\f{q}A$. We have a canonical Vietoris topology on $\f{S}_\r{clc}\c{S}A$ induced from the weak*-topology of $\c{S}A$.
In the case that $A$ is commutative it is proved in Section \ref{s2} that this Vietoris topology coincides with the vietoris topology
on the hyperspace of closed subsets of the Gelfand space of $A$. Suppose that $\f{q}A$ has a quantum metric space
structure in the sense of Rieffel \cite{Rieffel3,Rieffel1,Rieffel2}. This induces Hausdorff and infimum distances on $\f{S}_\r{cl}\f{q}A$.
Again if $A$ is commutative it is proved that these distances coincide with the usual Hausdorff and infimum distances 
(see the last paragraph of Section \ref{s4}).

The notion of \emph{quantum (or NC) metric space} have been considered by many authors, see
\cite{Rieffel2,KuperbergWeaver1,KerrLi1,GuidoIsola1,Latremoliere1,MartinettiMercatiTomassini1,Wu2,Sadr1} and references therein.
The main subjects studied in most of the mentioned papers are variations of the quantum Gromov-Hausdorff distance and quantum metric spaces
defined by Rieffel \cite{Rieffel2}. The notions introduced by Rieffel \cite{Rieffel2} are based on order unit spaces. Since our attention here
is to NC Topology we are more interested in order unit spaces arising from C*-algebras.

The plan of the paper is as follows. In Section \ref{s2} we consider some properties of hyperspaces associated to ordinary topological spaces. 
Also we consider Hausdorff and infimum distances. In Section \ref{s3} we review the notion of quantum metric space. In Section \ref{s4}
we introduce our main object $\f{S}_\r{cl}\f{q}A$, the hyperspace of closed subsets of a compact NC space.
In Section \ref{s5} we study the Vietoris topology on $\f{S}_\r{cl}\f{q}A$. In Section \ref{s6} using the infimum distance we define
an analog of Lipschitz seminorm for quantum metric spaces. At last in Section \ref{s7} we consider some questions and problems
on finite NC spaces and quantum tori.
%%%%%%%%%%%%%%%%%%%%%%%%%%%%%%%%%%%%%%%%%%%%%%%%%%%%%%%%%%%%%%%%%%%%%%%%%%%%%%%%%%%%%%%%%%%%%%%%%%%%%%%%%%%%%%%%%%%%%%%%%%
%%%%%%%%%%%%%%%%%%%%%%%%%%%%%%%%%%%%%%%%%%%%%%%%%%%%%%%%%%%%%%%%%%%%%%%%%%%%%%%%%%%%%%%%%%%%%%%%%%%%%%%%%%%%%%%%%%%%%%%%%%
%%%%%%%%%%%%%%%%%%%%%%%%%%%%%%%%%%%%%%%%%%%%%%%%%%%%%%%%%%%%%%%%%%%%%%%%%%%%%%%%%%%%%%%%%%%%%%%%%%%%%%%%%%%%%%%%%%%%%%%%%%
%%%%%%%%%%%%%%%%%%%%%%%%%%%%%%%%%%%%%%%%%%%%%%%%%%%%%%%%%%%%%%%%%%%%%%%%%%%%%%%%%%%%%%%%%%%%%%%%%%%%%%%%%%%%%%%%%%%%%%%%%%
\section{Hyperspace of closed subsets of an ordinary topological space}\label{s2}
Let $X$ be a compact Hausdorff space. We denote by $\f{S}_\r{cl}X$ the set of all nonempty closed subsets of $X$.
For every open $U\subseteq X$, let $U^-:=\{K\in\f{S}_\r{cl}X:K\cap U\neq\emptyset\}$ and $U^+:=\{K\in\f{S}_\r{cl}X:K\subseteq U\}$.
The smallest topology on $\f{S}_\r{cl}X$ containing all $U^\pm$'s is called Vietoris topology.
The space $\f{S}_\r{cl}X$ together with the Vietoris topology is called the hyperspace of closed sets in $X$.
It is easy to see (\cite[Exercise 3.12]{IllanesNadler1}) that the hyperspace is compact and Hausdorff. Also the subspace topology of $X$,
where $X$ is considered as a subspace of $\f{S}_\r{cl}X$ via the canonical embedding $x\mapsto\{x\}$,
coincides with the original topology of $X$. Let $\c{C}X$ denote the C*-algebra of complex valued continuous functions on $X$.
We always endow the state space $\c{S}\c{C}X$ of $\c{C}X$ with weak* topology. We also identify $\c{S}\c{C}X$ with the
space of Borel regular probability measures on $X$. Then the map $\delta:x\mapsto\delta_x$
is a homeomorphism from $X$ onto the space of pure states of $\c{C}X$ where $\delta_x$ denote the point mass measure concentrated at $x$.
For a nonempty closed subset $K$ of $X$ let $\c{F}_K$ denote the set of those measures $\mu$ in $\c{S}\c{C}X$ with $\r{Spt}(\mu)\subseteq K$.
Then $\c{F}_K$ is a weak*-closed face of $\c{S}\c{C}X$. Also note that $\c{F}_K$ is the weak* closed convex hull of $\{\delta_x:x\in K\}$.
\begin{proposition}\label{p0.5}
The map $\c{F}:K\mapsto\c{F}_K$ is a homeomorphism from the hyperspace $\f{S}_\r{cl}X$ onto a closed subspace of the hyperspace $\f{S}_\r{cl}\c{S}\c{C}X$.
\end{proposition}
\begin{proof}
Since both of the hyperspaces are compact Hausdorff spaces and $\c{F}$ is injective it is enough to show that $\c{F}$ is continuous.
Let $U,V$ be arbitrary open subsets of $\c{S}\c{C}X$. We must show that $\c{F}^{-1}(U^+)$ and $\c{F}^{-1}(V^-)$ are open in $\f{S}_\r{cl}X$.
Suppose that $K\in\c{F}^{-1}(U^+)$. Thus $\c{F}_K\subseteq U$. Since $\c{F}_K$ is convex it follows from \cite[Theorem 1.10]{Rudin1}
that there is a convex open subset $U_0$ of $\c{S}\c{C}X$ with $\c{F}_K\subseteq U_0\subseteq\overline{U_0}\subseteq U$.
We have $K\in(\delta^{-1}(U_0))^+\subseteq\c{F}^{-1}(U^+)$. Thus $\c{F}^{-1}(U^+)$ is open. Now suppose that $K\in\c{F}^{-1}(V^-)$.
Thus there are $\mu\in\c{F}_K\setminus V^\r{c}$ and open subset $W$ of $\c{S}\c{C}X$ such that $\mu\in W$ and $W\cap V^C=\emptyset$.
It follows that there exist $x_1,\ldots,x_n\in\r{Spt}\mu\subseteq K$, $t_1,\ldots,t_n>0$ with $\sum_{i=1}^nt_i=1$, and an open subset
$W_0$ of $\c{S}\c{C}X$, such that $\sum_{i=1}^nt_i\delta_{x_i}\in W_0\subseteq W$. Thus there are open subsets $O_1,\ldots,O_n$ of $X$
with $x_i\in O_i$, and with the property that if $y_i\in O_i$ then $\sum_{i=1}^nt_i\delta_{y_i}\in W_0$. We have
$K\in\cap_{i=1}^nO_i^-\subseteq\c{F}^{-1}(V^-)$. Thus $\c{F}^{-1}(V^-)$ is open. This completes the proof.
\end{proof}
Now suppose that $X$ is metrizable and let $d$ be a compatible metric on $X$.
The Hausdorff distance $\c{H}_d$ (associated to $d$) on $\f{S}_\r{cl}X$ is defined by
\begin{align*}
\c{H}_d(K,K')=\inf\{r>0:K\subseteq\r{Ball}(K',r),K'\subseteq\r{Ball}(K,r)\}\quad(K,K'\in\f{S}_\r{cl}X),
\end{align*}
where $\r{Ball}(K,r)=\{y\in X:d(x,y)<r,\exists x\in K\}$. It is well known that $\c{H}_d$ is a metric and the topology induced by $\c{H}_d$
coincides with the Vietoris topology (\cite[Theorem 3.1]{IllanesNadler1}). Also the mapping $x\mapsto\{x\}$ is an isometric embedding of $X$
into $\f{S}_\r{cl}X$. The Lipschitz seminorm $\c{L}_d$ for (self-adjoint) elements of $\c{C}X$ is defined by
\begin{align}\label{f2}
\c{L}_d(f):=\sup\{\frac{|f(x)-f(y)|}{d(x,y)}:x,y\in X, x\neq y\}\quad(f\in\c{C}X_\r{sa}).
\end{align}
This seminorm satisfies the Leibniz inequality: $\c{L}_d(fg)\leq\c{L}_d(f)\|g\|_\infty+\|f\|_\infty\c{L}_d(g)$.
The Lipschitz algebra of $(X,d)$ is defined by $\r{Lip}_dX:=\{f\in\c{C}X_\r{sa}:\c{L}_d(f)<\infty\}$. This is a real uniformly-dense subalgebra
of $\c{C}X_\r{sa}$. (For an extensive account on Lipschitz algebras see \cite{Weaver1}.)
The Monge-Kantorovich distance is defined by
\begin{align}\label{f3}
\rho_d(\mu,\nu):=\sup\{|\mu(f)-\nu(f)|:\c{L}_d(f)\leq1\}\quad(\mu,\nu\in\c{S}\c{C}X)
\end{align}
It is well known that the topology of $\rho_d$ coincides with weak* topology and also the restriction of $\rho_d$ to the space of pure states of
$\c{C}X$ is equal to $d$ where the pure state space is canonically identified with $X$. The metric version of Proposition \ref{p0.5} is as follows.
\begin{proposition}\label{p1}
$K\mapsto\c{F}_K$ is an isometric embedding from $(\f{S}_\r{cl}X,\c{H}_d)$ into $(\f{S}_\r{cl}\c{S}\c{C}X,\c{H}_{\rho_d})$.
\end{proposition}
\begin{proof}
Let $h$ denote the Hausdorff distance of $\c{F}_K$ and $\c{F}_{K'}$. Suppose that $\c{H}_d(K,K')<r$. Then for every $x\in K$ there is $y\in K'$
such that $d(x,y)<r$. Let $t_1,\ldots,t_n\geq0$ with $\sum_{i=1}^nt_i=1$ and let $x_1,\ldots,x_n\in K$. Then it is easily verified that
$\rho_d(\sum_{i=1}^nt_i\delta_{x_i},\sum_{i=1}^nt_i\delta_{y_i})<r$ where $y_i\in K'$ is such that $d(x_i,y_i)<r$. This shows that
$\c{F}_K\subseteq\r{Ball}(\c{F}_{K'},r+\epsilon)$ for every $\epsilon>0$. Similarly, we have $\c{F}_{K'}\subseteq\r{Ball}(\c{F}_K,r+\epsilon)$.
Thus $h\leq r$, and hence $h\leq\c{H}_d(K,K')$. Now suppose that $h<s$. Let $x\in K$. Then there are $z_1,\ldots,z_n\in K'$ and
$t_1,\ldots,t_n\geq0$ with $\sum_{i=1}^nt_i=1$ and $\rho_d(\delta_x,\sum_{i=1}^nt_i\delta_{z_i})<s$. Let $f$ be the function on $X$ defined by
$y\mapsto d(x,y)$. Then $\c{L}_d(f)=1$ (if $X$ at least has two points). We have
$$\sum_{i=1}^nt_id(x,z_i)=|\delta_x(f)-(\sum_{i=1}^nt_i\delta_{z_i})(f)|<s.$$
Thus $d(x,z_{i_0})<s$ for some $i_0$. This shows that $K\subseteq\r{Ball}(K',s)$. Similarly we have $K'\subseteq\r{Ball}(K,s)$. Thus  $\c{H}_d(K,K')\leq s$,
and hence $\c{H}_d(K,K')\leq h$. The proof is complete.
\end{proof}
For two subsets $K,K'$ of $X$ their infimum distance is defined by
\begin{align*}
\c{I}_d(K,K'):=\inf\{d(x,y):x\in K,y\in K'\}.
\end{align*}
In the case that $K$ or $K'$ is empty we let $\c{I}_d(K,K')=\infty$. Note also that in general $\c{I}_d$ is not a metric on $\f{S}_\r{cl}X$.
\begin{proposition}\label{p2}
Let $K,K'\in\f{S}_\r{cl}X$. Then $\c{I}_d(K,K')=\c{I}_{\rho_d}(\c{F}_K,\c{F}_{K'})$.
\end{proposition}
\begin{proof}
Let $I$ denote the infimum distance of $\c{F}_K$ and $\c{F}_{K'}$.
For $x\in K,y\in K'$ we have $\delta_x\in\c{F}_K,\delta_y\in\c{F}_{K'}$ and $d(x,y)=\rho_d(\delta_x,\delta_y)$. Thus $I\leq\c{I}_d(K,K')$.
Let $I<r$. There are $\mu:=\sum_{i=1}^nt_i\delta_{x_i}\in\c{F}_K$ and $\nu:=\sum_{j=1}^ms_j\delta_{y_j}\in\c{F}_{K'}$ such that $\rho_d(\mu,\nu)<r$.
Let the function $f$ on $X$ be defined by $x\mapsto\c{I}_d(\{x\},\{y_1,\ldots,y_n\})$. Then $\c{L}_d(f)=1$, and we have
$$\sum_{i=1}^nt_if(x_i)=|\mu(f)-\nu(f)|<r.$$
Thus $f(x_{i_0})<r$ for some $i_0$, and hence there is $j_0$ such that $d(x_{i_0},y_{j_0})<r$. This shows that $\c{I}_d(K,K')<r$.
Since $r>I$ is arbitrary we conclude that $\c{I}_d(K,K')\leq I$.
\end{proof}
It is well known that
\begin{align}\label{f5}
\c{L}_d(f)=\sup\{\frac{|\mu(f)-\nu(f)|}{\rho_d(\mu,\nu)}:\mu,\nu\in\c{S}\c{C}X, \mu\neq \nu\}\quad(f\in\c{C}X_\r{sa}).
\end{align}
Also the following formula follows from (\ref{f2}) and Proposition \ref{p2}.
\begin{align}\label{f6}
\c{L}_d(f)=\sup_{\lambda<\lambda'\in\mathbb{R}}\frac{\lambda'-\lambda}{\c{I}_{d}(f^{-1}\lambda',f^{-1}\lambda)}
=\sup_{\lambda<\lambda'\in\mathbb{R}}\frac{\lambda'-\lambda}{\c{I}_{\rho_d}(\c{F}_{f^{-1}\lambda'},\c{F}_{f^{-1}\lambda})}
\quad(f\in\c{C}X_\r{sa}).
\end{align}
%%%%%%%%%%%%%%%%%%%%%%%%%%%%%%%%%%%%%%%%%%%%%%%%%%%%%%%%%%%%%%%%%%%%%%%%%%%%%%%%%%%%%%%%%%%%%%%%%%%%%%%%%%%%%%%%%%%%%%%%%%
%%%%%%%%%%%%%%%%%%%%%%%%%%%%%%%%%%%%%%%%%%%%%%%%%%%%%%%%%%%%%%%%%%%%%%%%%%%%%%%%%%%%%%%%%%%%%%%%%%%%%%%%%%%%%%%%%%%%%%%%%%
%%%%%%%%%%%%%%%%%%%%%%%%%%%%%%%%%%%%%%%%%%%%%%%%%%%%%%%%%%%%%%%%%%%%%%%%%%%%%%%%%%%%%%%%%%%%%%%%%%%%%%%%%%%%%%%%%%%%%%%%%%
%%%%%%%%%%%%%%%%%%%%%%%%%%%%%%%%%%%%%%%%%%%%%%%%%%%%%%%%%%%%%%%%%%%%%%%%%%%%%%%%%%%%%%%%%%%%%%%%%%%%%%%%%%%%%%%%%%%%%%%%%%
\section{Compact quantum metric spaces}\label{s3}
For the theory of order unit spaces we refer the reader to \cite{Alfsen1}. We denote the state space of an order unit space $B$ by $\c{S}B$.
This space is always considered with the weak* topology.
Let $A$ be a unital C*-algebra with the self-adjoint part $A_\r{sa}$ and state space $\c{S}A$.
Suppose that $B$ is any real linear subspace of $A_\r{sa}$ that contains $1_A$.
Then $B$ together with the usual partial ordering between self-adjoint elements and $1_A$ as order unit
becomes an order unit space. Moreover if $B$ is dense in $A_\r{sa}$ (with the norm topology) then the mapping $\mu\mapsto\mu|_B$
defines an affine homeomorphism from $\c{S}A$ with the weak* topology onto $\c{S}B$.

It is clear that the state space of an order unit space with weak* topology is a compact convex subset of a locally convex Hausdorff space.
The converse of this fact is also well known (see the details after Corollary II.2.3 of \cite{Alfsen1}).
Indeed, let $E$ be a compact convex subset of a locally convex Hausdorff space;
if $\c{A}E$ denotes the order unit space of all continuous affine real valued functions on $E$ with the constant function
$1_E$ as order unit, then $E$ and $\c{S}\c{A}E$ are affinely  homeomorphic via the map $x\mapsto(f\mapsto f(x))$ ($x\in E$).
Thus there is no difference that we formulate our results in terms of order unit spaces or else using compact convex sets.

Let $B$ be an order unit space and $L$ be a seminorm on $B$. By analogy with Formula (\ref{f3}) we define a pseudo-metric on $\c{S}B$ as follows.
\begin{align}\label{f4}
\rho_L(\mu,\nu):=\sup\{|\mu(b)-\nu(b)|:L(b)\leq1\}\quad(\mu,\nu\in\c{S}B)
\end{align}
Note that in general $\rho_L$ does not separate the points and may take value $+\infty$.
\begin{definition}
(\cite{Rieffel3,Rieffel1,Rieffel2})
Let $B$ be an order unit space with order unit $e$ and let $L$ be a seminorm on $B$
satisfying the following two conditions:
\begin{enumerate}
\item[i)] $L(b)=0$ if and only if $b=\lambda e$ for some $\lambda\in\mathbb{R}$.
\item[ii)] The topology induced by $\rho_L$, given by Formula (\ref{f4}), coincides with the weak*-topology on $\c{S}B$.
\end{enumerate}
Then the pair $(B,L)$ is called a compact quantum metric space. Also if a unital C*-algebra $A$ is given such that $B$
as an order unit space is a subspace of $A_\r{sa}$ containing $1_A$, and $B$ is dense in $A_\r{sa}$ w.r.t. the C*-norm, 
then $(A,B,L)$ is called a C*-algebraic compact quantum metric space. In the case that $(B,L)$ is understood 
we say that $\f{q}A$ is a C*-algebraic quantum metric space.
\end{definition}
Let $(X,d)$ be an ordinary compact metric space. Then $(\c{C}X,\r{Lip}_dX,\c{L}_d)$
is a C*-algebraic compact quantum metric space. Also by (\ref{f3}) and (\ref{f4}) we have $\rho_{\c{L}_d}=\rho_d$.
Thus the structure of $(X,d)$ is completely recovered by $(\c{C}X,\r{Lip}_dX,\c{L}_d)$.
We remark that there are examples of C*-algebraic compact quantum metric spaces
$(A,B,L)$ with $A=\c{C}X$ for a compact space $X$ such that $L$ does not arise from any ordinary metric $d$
on $X$ i.e. $L\neq\c{L}_d$, see \cite[Example 7.1 and Theorem 8.1]{Rieffel3}. For other examples of nonclassical quantum metric spaces we refer the reader
to the list of papers in Introduction.
%%%%%%%%%%%%%%%%%%%%%%%%%%%%%%%%%%%%%%%%%%%%%%%%%%%%%%%%%%%%%%%%%%%%%%%%%%%%%%%%%%%%%%%%%%%%%%%%%%%%%%%%%%%%%%%%%%%%%%%%%%
%%%%%%%%%%%%%%%%%%%%%%%%%%%%%%%%%%%%%%%%%%%%%%%%%%%%%%%%%%%%%%%%%%%%%%%%%%%%%%%%%%%%%%%%%%%%%%%%%%%%%%%%%%%%%%%%%%%%%%%%%%
%%%%%%%%%%%%%%%%%%%%%%%%%%%%%%%%%%%%%%%%%%%%%%%%%%%%%%%%%%%%%%%%%%%%%%%%%%%%%%%%%%%%%%%%%%%%%%%%%%%%%%%%%%%%%%%%%%%%%%%%%%
%%%%%%%%%%%%%%%%%%%%%%%%%%%%%%%%%%%%%%%%%%%%%%%%%%%%%%%%%%%%%%%%%%%%%%%%%%%%%%%%%%%%%%%%%%%%%%%%%%%%%%%%%%%%%%%%%%%%%%%%%%
\section{Hyperspace of closed sets in a NC space}\label{s4}
Let $A$ be a unital C*-algebra with the state space $\c{S}A$.
Let $A''$ denote the second commutant of $A$ in the universal representation of $A$.
By the Sherman Theorem the second dual $A^{**}$ is canonically isomorphic to the von Neumann algebra $A''$ where
$A^{**}$ is considered as a C*-algebra with the Arens product. A projection $p\in A^{**}$ is called closed \cite{Akemann1}
if there is a decreasing net of positive elements of $A$ that converges to $p$ in the weak* topology. A projection $q$ is called open if
$1-q$ is closed. For every projection $p\in A^{**}$ we let $\c{F}_p:=\{\mu\in \c{S}A:\sl \mu,p\sr=1\}$.
In the case that $A=\c{C}X$ for a compact Hausdorff space $X$, there is a bijection $K\mapsto p_K$ between closed subsets $K$ of $X$
and closed projections of $A^{**}$ such that $\c{F}_K=\c{F}_{p_K}$ with the notations of Section \ref{s2}.
\begin{proposition}\label{p3}
The assignment $p\mapsto\c{F}_p$ is a bijection between closed projections of $A^{**}$ and weak*-closed faces of $\c{S}A$.
\end{proposition}
\begin{proof}
It follows from \cite[Theorems 3.6.11 and 3.10.7]{Pedersen1} or \cite[Theorem 2.5]{AkemannPedersen1}.
\end{proof}
Let $E$ be a compact convex subset of a locally convex Hausdorff space.
We denote by $\f{S}_\r{clc}E$ the set of nonempty closed convex subsets of $E$, and by $\f{S}_\r{clf}E$ the set of all closed faces of $E$.
Thus we have the chain $\f{S}_\r{clf}E\subset\f{S}_\r{clc}E\subset\f{S}_\r{cl}E$ of Hyperspaces. Throughout the paper these hyperspaces
are endowed with Vietoris topology.

Let $B$ be an order unit space. In \cite{Rieffel2} (imaginary) closed subsets of the quantum space $\f{q}B$
are identified with elements of $\f{S}_\r{clc}\c{S}B$. As we saw above it is more natural to consider the closed subsets as elements of
$\f{S}_\r{clf}\c{S}B$. So by analogy with the notations of
Section \ref{s2} we would use the symbol $\f{S}_\r{cl}\f{q}B$ instead of $\f{S}_\r{clf}\c{S}B$. Analogously, for a unital C*-algebra $A$
we let $\f{S}_\r{cl}\f{q}A:=\f{S}_\r{clf}\c{S}A$. In the case that $A=\c{C}X$ for a compact Hausdorff space $X$ it follows from
Proposition \ref{p0.5} that $\f{S}_\r{cl}\f{q}A$ is homeomorphic to $\f{S}_\r{cl}X$. Suppose that $X$ has a compatible metric $d$ and
consider the C*-algebraic quantum metric space $(\c{C}X,\r{Lip}_dX,\c{L}_d)$. Let $\rho:=\rho_{\c{L}_d}=\rho_d$.
It follows from Proposition \ref{p1} that the metric spaces $(\f{S}_\r{cl}\f{q}A,\c{H}_\rho)$ and $(\f{S}_\r{cl}X,\c{H}_d)$
are isometrically isomorphic. Also it follows from Proposition \ref{p2} that the distance functions $\c{I}_\rho$ on $\f{S}_\r{cl}\f{q}A$
and $\c{I}_d$ on $\f{S}_\r{cl}X$ coincide when the two spaces are considered canonically identical.
%%%%%%%%%%%%%%%%%%%%%%%%%%%%%%%%%%%%%%%%%%%%%%%%%%%%%%%%%%%%%%%%%%%%%%%%%%%%%%%%%%%%%%%%%%%%%%%%%%%%%%%%%%%%%%%%%%%%%%%%%%
%%%%%%%%%%%%%%%%%%%%%%%%%%%%%%%%%%%%%%%%%%%%%%%%%%%%%%%%%%%%%%%%%%%%%%%%%%%%%%%%%%%%%%%%%%%%%%%%%%%%%%%%%%%%%%%%%%%%%%%%%%
%%%%%%%%%%%%%%%%%%%%%%%%%%%%%%%%%%%%%%%%%%%%%%%%%%%%%%%%%%%%%%%%%%%%%%%%%%%%%%%%%%%%%%%%%%%%%%%%%%%%%%%%%%%%%%%%%%%%%%%%%%
%%%%%%%%%%%%%%%%%%%%%%%%%%%%%%%%%%%%%%%%%%%%%%%%%%%%%%%%%%%%%%%%%%%%%%%%%%%%%%%%%%%%%%%%%%%%%%%%%%%%%%%%%%%%%%%%%%%%%%%%%%
\section{Vietoris topology}\label{s5}
Throughout this section $E$ denotes a compact convex subset of a locally convex Hausdorff space.
The following result stated as Theorem \ref{t1} is very well known, at least in the case that $E$ is metrizable;
but we did not find in literatures any proof for the general case; however its proof is easy and based
on the following lemma. (Let $\Lambda,\Lambda'$ be directed sets and $(x_\lambda)_{\lambda\in\Lambda}$ be a net in $X$.
Let $f:\Lambda'\to\Lambda$ be an order preserving function such that $\forall\lambda\in\Lambda,\exists\lambda'\in\Lambda':f(\lambda')\geq\lambda$.
Then the net $(x_{f(\lambda')})_{\lambda'\in\Lambda'}$ is called a subnet of $(x_\lambda)_{\lambda\in\Lambda}$.)
\begin{lemma}\label{l1}
Let $X$ be a compact Hausdorff space and $(K_\lambda)_\lambda$ a net in $\f{S}_\r{cl}X$ converging to $K$.
\begin{enumerate}
\item[(i)] If $(x_\lambda)_\lambda$ is a net in $X$ such that $x_\lambda\to x$ and $x_\lambda\in K_\lambda$, then $x\in K$.
\item[(ii)] If $x\in K$, then there are a subnet $(K_{\lambda'})_{\lambda'}$ of $(K_\lambda)_\lambda$
and a net $(x_{\lambda'})_{\lambda'}$ such that $x_{\lambda'}\in K_{\lambda'}$ and $x_{\lambda'}\to x$.
\end{enumerate}
\end{lemma}
\begin{proof}
Straightforward.
\end{proof}
\begin{theorem}\label{t1}
The hyperspace $\f{S}_\r{clc}E$ is a compact Hausdorff space.
\end{theorem}
\begin{proof}
It is enough to show that $\f{S}_\r{clc}E$ is a closed subset of $\f{S}_\r{cl}E$.
Suppose that $(K_\lambda)_\lambda$ is a net in $\f{S}_\r{clc}E$ converging to $K\in\f{S}_\r{cl}E$. We must show that $K$ is convex.
Suppose that $x,y\in K$ and $0\leq t\leq1$. By Lemma \ref{l1}(ii) there exist a subnet $(K_{\lambda'})_{\lambda'}$ of $(K_\lambda)_\lambda$
and nets $(x_{\lambda'})_{\lambda'},(y_{\lambda'})_{\lambda'}$ such that $x_{\lambda'},y_{\lambda'}\in K_{\lambda'}$ and $x_{\lambda'}\to x,y_{\lambda'}\to y$.
Thus $(tx_{\lambda'}+(1-t)y_{\lambda'})_{\lambda'}$ is a net in $K_{\lambda'}$ converging to $tx+(1-t)y$. Now it follows from Lemma \ref{l1}(i) that
$tx+(1-t)y\in K$. The proof is complete.
\end{proof}
For metrizable $E$ we have the following strong result of Nadler-Quinn-Stavrakos:
\begin{theorem}\label{tHilbertCube}
Suppose that $E$ is metrizable and the real dimension of the smallest real hyperplane containing $E$ is $\geq2$.
Then $\f{S}_\r{clc}E$ is homeomorphic to Hilbert cube.
\end{theorem}
\begin{proof}
This is a restatement of \cite[Theorem 2.2]{NadlerQuinnStavrakos1}. (Note that in the proof of \cite[Theorem 2.2]{NadlerQuinnStavrakos1}
it is enough that $K$ be metrizable.)
\end{proof}
For some results similar to Theorem \ref{tHilbertCube} in the case that $E$ is not metrizable see \cite{BazylevychRepovsZarichnyi1}.
A direct consequence of Theorem \ref{tHilbertCube} is the following.
\begin{corollary}\label{c1}
Let $(B,L)$ be a compact quantum metric space such that the (real vector space) dimension of $B$ is $\geq2$. Then $\f{S}_\r{clc}\c{S}B$
is homeomorphic to Hilbert cube. In particular, if $\f{q}A$ is a C*-algebraic quantum metric space
such that $A\neq0,\mathbb{C}$ then $\f{S}_\r{clc}\c{S}A$ is homeomorphic to Hilbert cube.
\end{corollary}
Let $\partial_\r{e}E$ denote the subspace of extreme points of $E$.
\begin{theorem}\label{t2}
If the hyperspace $\f{S}_\r{clf}E$ is compact then $\partial_\r{e}E$ is compact.
\end{theorem}
\begin{proof}
Suppose that $\f{S}_\r{clf}E$ is compact. We must show that $\partial_\r{e}E$ is a closed subset of $E$.
Let $(e_\lambda)_\lambda$ be a net in $\partial_\r{e}E$ converging to $x\in E$. Since $e_\lambda$ is an extreme point
$\{e_\lambda\}$ is a closed face of $E$. Thus there is a subnet $(e_{\lambda'})_{\lambda'}$ such that $\{e_{\lambda'}\}\to K$
in $\f{S}_\r{clf}E$. Now it follows from Lemma \ref{l1} that $K=\{x\}$ which means that $x\in\partial_\r{e}E$. Thus
$\partial_\r{e}E$ is a closed subset of $E$.
\end{proof}
In general the converse of Theorem \ref{t2} is not satisfied even if $E$ is finite dimensional, see \cite{ReiterStavraks1}.
We say that $E$ is stable \cite{Papadopoulou1} if for every $0\leq t\leq1$ the map $(x,y)\mapsto tx+(1-t)y$ from $E\times E$ into $E$
is open. Among examples of stable compact convex sets are Bauer simplices \cite[Theorem 1]{Obrien1}.
For a complete account on Bauer simplices see \cite{Alfsen1}. It is well known that a unital C*-algebra is commutative if and only if
its state space with weak*-topology is a Bauer simplex \cite[Remark in page 296]{AlfsenShultz1},\cite{Batty1}.
To our knowledge the following result has not been mentioned before in the literatures.
\begin{theorem}\label{t3}
If $E$ is stable then $\f{S}_\r{clf}E$ is compact.
\end{theorem}
\begin{proof}
Suppose that $E$ is stable. We must show that $\f{S}_\r{clf}E$ is a closed subset of $\f{S}_\r{clc}E$.
Let $(K_\lambda)_\lambda$ be a net in $\f{S}_\r{clf}E$ converging to $K\in\f{S}_\r{clc}E$. We show that $K$ is a face of $E$.
Suppose that for some $0\leq t\leq1$ and $x,y\in E$ we have $z:=tx+(1-t)y\in K$. By Lemma \ref{l1}(ii) there exist a subnet $(K_{\lambda'})_{\lambda'}$
of $(K_\lambda)_\lambda$ and a net $(z_{\lambda'})_{\lambda'}$ such that $z_{\lambda'}\in K_{\lambda'}$ and $z_{\lambda'}\to z$. Since the map
$\phi:(x',y')\mapsto tx'+(1-t)y'$ is open and continuous, for every open $U$ in $E$ containing $z$ there are opens $V,W$ respectively
containing $x,y$ such that $\phi(V\times W)$ is an open in $U$. This property enables us to find a subnet $(z_{\lambda''})_{\lambda''}$
of $(z_{\lambda'})_{\lambda'}$ and nets $(x_{\lambda''})_{\lambda''},(y_{\lambda''})_{\lambda''}$ such that $x_{\lambda''}\to x,y_{\lambda''}\to y$
and $z_{\lambda''}=tx_{\lambda''}+(1-t)y_{\lambda''}$. Since $K_{\lambda''}$ is a face we have $x_{\lambda''},y_{\lambda''}\in K_{\lambda''}$.
Now it follows from Lemma \ref{l1}(i) that $x,y\in K$. The proof is complete.
\end{proof}
A direct consequence of Theorems \ref{t2} and \ref{t3} is the following result.
\begin{corollary}
Let $A$ be a unital C*-algebra. If $\f{S}_\r{cl}\f{q}A$ is compact then the space of pure states of $A$ is weak*-compact.
If $\c{S}A$ is stable then $\f{S}_\r{cl}\f{q}A$ is compact.
\end{corollary}
%%%%%%%%%%%%%%%%%%%%%%%%%%%%%%%%%%%%%%%%%%%%%%%%%%%%%%%%%%%%%%%%%%%%%%%%%%%%%%%%%%%%%%%%%%%%%%%%%%%%%%%%%%%%%%%%%%%%%%%%%%
%%%%%%%%%%%%%%%%%%%%%%%%%%%%%%%%%%%%%%%%%%%%%%%%%%%%%%%%%%%%%%%%%%%%%%%%%%%%%%%%%%%%%%%%%%%%%%%%%%%%%%%%%%%%%%%%%%%%%%%%%%
%%%%%%%%%%%%%%%%%%%%%%%%%%%%%%%%%%%%%%%%%%%%%%%%%%%%%%%%%%%%%%%%%%%%%%%%%%%%%%%%%%%%%%%%%%%%%%%%%%%%%%%%%%%%%%%%%%%%%%%%%%
%%%%%%%%%%%%%%%%%%%%%%%%%%%%%%%%%%%%%%%%%%%%%%%%%%%%%%%%%%%%%%%%%%%%%%%%%%%%%%%%%%%%%%%%%%%%%%%%%%%%%%%%%%%%%%%%%%%%%%%%%%
\section{Infimum distance and an analog of Lipschitz seminorm}\label{s6}
Let $A$ be a unital C*-algebra  and let $\rho$ be a compatible metric on $\c{S}A$.
Let $L_1:A_\r{sa}\to[0,\infty]$ be a seminorm given by the analog of Formula (\ref{f5}):
\begin{align*}
L_1(a)=\sup\{\frac{|\mu(a)-\nu(a)|}{\rho(\mu,\nu)}:\mu,\nu\in\c{S}A, \mu\neq \nu\}.
\end{align*}
Let $H$ denote the Hilbert space of the universal representation of $A$, and $\c{B}H$ be the algebra of bounded operators on $H$.
Then by definition we have $A\subset A''\subseteq \c{B}H$. For $a\in A_\r{sa}$ let $E_a$ denotes the spectral measure of $a$ defined on the
Borel subsets of $\bb{R}$ where $a$ is considered as an element of $\c{B}H$. It is well known that for every closed subset $S$ of $\bb{R}$
the projection $E_a(S)$ is a closed projection in $A''$. (The converse is also true \cite[Theorem A1]{Akemann2},\cite{Akemann3}, that is if
$a\in A_\r{sa}''$ and $E_a(S)$ is a closed projection for every closed subset $S\subseteq\bb{R}$ then $a\in A$.)
For $a\in A_\r{sa}$ and $\lambda\in\bb{R}$ let $\c{F}_{a,\lambda}$ denote the weak*-closed face of $\c{S}A$ corresponding to the closed projection
$E_a(\{\lambda\})$ as in Proposition \ref{p3}. (In the case that $E_a(\{\lambda\})=0$ we let $\c{F}_{a,\lambda}=\emptyset$.)
In Section \ref{s2} we restated the definition of Lipschitz seminorm for an ordinary metric space as Formula (\ref{f6}). 
Now analogously we define a function $L_2:A_\r{sa}\to[0,\infty]$ by
\begin{align*}
L_2(a)=\sup_{\lambda<\lambda'\in\mathbb{R}}\frac{\lambda'-\lambda}{\c{I}_{\rho}(\c{F}_{a,\lambda'},\c{F}_{a,\lambda})}.
\end{align*}
We have $L_2\leq L_1$ but in general $L_2$ is not a seminorm.
\begin{question}
Under which conditions is $L_2$ a seminorm (with Leibniz property) on any commutative subalgebra of $A$?  
\end{question} 
Suppose $\rho$ is induced by a C*-algebraic quantum metric structure $(A,B,L)$ i.e. $\rho=\rho_L$. Then we have $L_2(a)\leq L_1(a)\leq L(a)$ for $a\in B$.
As we saw in Section \ref{s2} in the classical case $(A,B,L)=(\c{C}X,\r{Lip}_dX,\c{L}_d)$ we have $L_2=L_1=L$.
By Theorem 4.1 of \cite{Rieffel3} we know that if $L$ is lower semicontinuous (which means $\{a\in B:L(a)\leq1\}$ is closed in $B$
w.r.t. the C*-norm) then $L_1(a)=L(a)$ for every $a\in B$.
%%%%%%%%%%%%%%%%%%%%%%%%%%%%%%%%%%%%%%%%%%%%%%%%%%%%%%%%%%%%%%%%%%%%%%%%%%%%%%%%%%%%%%%%%%%%%%%%%%%%%%%%%%%%%%%%%%%%%%%%%%
%%%%%%%%%%%%%%%%%%%%%%%%%%%%%%%%%%%%%%%%%%%%%%%%%%%%%%%%%%%%%%%%%%%%%%%%%%%%%%%%%%%%%%%%%%%%%%%%%%%%%%%%%%%%%%%%%%%%%%%%%%
%%%%%%%%%%%%%%%%%%%%%%%%%%%%%%%%%%%%%%%%%%%%%%%%%%%%%%%%%%%%%%%%%%%%%%%%%%%%%%%%%%%%%%%%%%%%%%%%%%%%%%%%%%%%%%%%%%%%%%%%%%
%%%%%%%%%%%%%%%%%%%%%%%%%%%%%%%%%%%%%%%%%%%%%%%%%%%%%%%%%%%%%%%%%%%%%%%%%%%%%%%%%%%%%%%%%%%%%%%%%%%%%%%%%%%%%%%%%%%%%%%%%%
\section{Some questions and problems}\label{s7}
We saw that for a unital commutative C*-algebra $A$, $\f{S}_\r{cl}\f{q}A$ is compact.
\begin{problem}
Characterize those unital C*-algebras $A$ such that $\f{S}_\r{cl}\f{q}A$ is compact.
\end{problem}
\begin{question}
For which C*-algebras $A$, is $\f{S}_\r{cl}\f{q}A$ (path or locally path) connected? (See \cite{BanakhVoytsitskyy1} in the classical case.)
\end{question}
Let $\c{M}_n$ denote the C*-algebra of $n\times n$ matrixes. In NC Geometry $\f{q}\c{M}_n$ is usually considered as the
finite NC space with $n$ points. Since $\c{M}_n=\c{M}^{**}$ any projection in $\c{M}_n$ is closed and open
(\cite[Proposition II.18]{Akemann1}), and hence $\f{S}_\r{cl}\f{q}\c{M}_n$ as a set is canonically identified with $\cup_{i=1}^n\r{Gr}(i,n)$
where $\r{Gr}(i,n)$ denote the Grassmannian manifold of $i$-dimensional subspaces of $\bb{C}^n$.
\begin{question}
Is the subspace of $\f{S}_\r{cl}\f{q}\c{M}_n$ containing projections of rank $i$ homeomorphic to $\r{Gr}(i,n)$?
\end{question}
Let $X$ be a compact Hausdorff space and $C$ be a C*-subalgebra of $\c{C}X$ containing $1_X$.
Let $Z$ denote the pure state space of $C$ with weak*-topology. We have a canonical continuous surjective
map $\Gamma:X\to Z$ defined by $\Gamma(x)(c)=c(x)$ ($c\in C$). It is easily checked that the topology of $Z$ is the quotient
topology under $\Gamma$. Also $\Gamma$ induces the family $\{K_z\}_{z\in Z}$ of nonempty disjoint closed subsets of $X$ parameterized by $Z$
where $K_z:=\Gamma^{-1}(z)$. A generalization of this notion is as follows.
\begin{definition}\label{d1}
Let $A$ be a unital C*-algebra and $C$ be a C*-subalgebra of $A$ containing the unit.
Let $Z$ denote the pure state space of $C$. For every $z\in Z$ let $\c{F}_z:=\{\mu\in\c{S}A:\mu(c)=z(c), c\in C\}$. Then $\c{F}_z$
is a weak*-closed face of $\c{S}A$. We say that $\{\c{F}_z\}_{z\in Z}$ is the family of closed subsets of $\f{q}A$ parameterized by $\f{q}C$.
\end{definition}
Let $0\leq\theta<1$. The quantum torus $\bb{T}^2_\theta:=\f{q}\c{C}\bb{T}^2_\theta$ is
the NC space associated to the universal C*-algebra $\c{C}\bb{T}^2_\theta$
generated by two unitary elements $u,v$ satisfying $uv=e^{2\pi i\theta}vu$. Let $\bb{T}:=\{z\in\bb{C}:|z|=1\}$ denote the unit circle.
We identify the C*-subalgebra generated by $v$ with $\c{C}\bb{T}$ via the *-isomorphism given by the assignment $v\mapsto\r{id}_\bb{T}$
where $\r{id}_\bb{T}\in\c{C}\bb{T}$ is the identity function. For every $z\in\bb{T}$ let
$\bb{T}_{\theta,z}:=\{\mu\in\c{S}\c{C}\bb{T}^2_\theta:\mu(f)=f(z), f\in \c{C}\bb{T}\}$. Then we call $\{\bb{T}_{\theta,z}\}_{z\in\bb{T}}$ the family
of $v$-sub-circles in $\bb{T}^2_\theta$. The name is justified as follows. It is clear that $\c{C}\bb{T}^2_0$ can be identified
with $\c{C}\bb{T}^2$ via the *-isomorphism given by the assignments $u\mapsto\r{id}_1,v\mapsto\r{id}_2$ where $\r{id}_1,\r{id}_2\in\c{C}\bb{T}^2$
are respectively the projection functions on the first and second components of $\bb{T}^2=\bb{T}\times\bb{T}$. Then $\bb{T}_{0,z}$
is identified with the set of Borel probability measures $\mu$ on $\bb{T}^2$ such that the support of $\mu$ is contained in the sub-circle
$\{(w,z)\in\bb{T}^2:w\in\bb{T}\}$.

It is not hard to see that the map $z\mapsto\{(w,z)\in\bb{T}^2:w\in\bb{T}\}$ from $\bb{T}$ into $\f{S}_\r{cl}\bb{T}^2$ is continuous
with Vietoris topology. So it is natural to ask the following questions.
\begin{question}
Is the map $z\mapsto\bb{T}_{\theta,z}$ ($\theta\neq0$) from $\bb{T}$ into $\f{S}_\r{cl}\bb{T}^2_\theta$ continuous?
Is the family of $v$-sub-circles in $\bb{T}^2_\theta$ compact or (path) connected?
\end{question}
If we have a (Riemannian) metric on $\bb{T}^2$ we can ask about the Hausdorff and infimum distances of sub-circles.
Analogously we have the following problem.
\begin{problem}
Consider $\bb{T}^2_\theta$ as a C*-algebraic quantum metric space described in \cite{Rieffel1,Rieffel2} and find the Hausdorff and
infimum distances between two arbitrary sub-circles $\bb{T}_{\theta,z}$ and $\bb{T}_{\theta,z'}$.
\end{problem}
%%%%%%%%%%%%%%%%%%%%%%%%%%%%%%%%%%%%%%%%%%%%%%%%%%%%%%%%%%%%%%%%%%%%%%%%%%%%%%%%%%%%%%%%%%%%%%%%%%%%%%%%%%%%%%%%%%%%%%%%%%
%%%%%%%%%%%%%%%%%%%%%%%%%%%%%%%%%%%%%%%%%%%%%%%%%%%%%%%%%%%%%%%%%%%%%%%%%%%%%%%%%%%%%%%%%%%%%%%%%%%%%%%%%%%%%%%%%%%%%%%%%%
%%%%%%%%%%%%%%%%%%%%%%%%%%%%%%%%%%%%%%%%%%%%%%%%%%%%%%%%%%%%%%%%%%%%%%%%%%%%%%%%%%%%%%%%%%%%%%%%%%%%%%%%%%%%%%%%%%%%%%%%%%
%%%%%%%%%%%%%%%%%%%%%%%%%%%%%%%%%%%%%%%%%%%%%%%%%%%%%%%%%%%%%%%%%%%%%%%%%%%%%%%%%%%%%%%%%%%%%%%%%%%%%%%%%%%%%%%%%%%%%%%%%%
\bibliographystyle{amsplain}

%%%%%%%%%%%%%%%%%%%%%%%%%%%%%%%%%%%%%%%%%%%%%%%%%%%%%%%%%%%%%%%%%%%%%%%%%%%%%%%%%%%%%%%%%%%%%%%%%%%%%%%%%%%%%%%%%%%%%%%%%%
%%%%%%%%%%%%%%%%%%%%%%%%%%%%%%%%%%%%%%%%%%%%%%%%%%%%%%%%%%%%%%%%%%%%%%%%%%%%%%%%%%%%%%%%%%%%%%%%%%%%%%%%%%%%%%%%%%%%%%%%%%
%%%%%%%%%%%%%%%%%%%%%%%%%%%%%%%%%%%%%%%%%%%%%%%%%%%%%%%%%%%%%%%%%%%%%%%%%%%%%%%%%%%%%%%%%%%%%%%%%%%%%%%%%%%%%%%%%%%%%%%%%%
%%%%%%%%%%%%%%%%%%%%%%%%%%%%%%%%%%%%%%%%%%%%%%%%%%%%%%%%%%%%%%%%%%%%%%%%%%%%%%%%%%%%%%%%%%%%%%%%%%%%%%%%%%%%%%%%%%%%%%%%%%
\end{document}